\documentclass{article}
\usepackage{graphicx} 
\usepackage{amsmath,ulem,tikz-cd}
\usepackage{xcolor}
\usepackage{comment}
\usepackage{hyperref}
\usepackage{amssymb}
\usepackage{mathrsfs} 
\usepackage{textcomp}
\usepackage{hyperref}
\usepackage[english]{babel}
\usepackage{amsthm}
\theoremstyle{definition}
\newtheorem{definition}{Definition}[section]
\newtheorem*{remark}{Remark}
\newtheorem*{outline}{Outline}
\newtheorem*{notation}{Notation}
\newtheorem{theorem}{Theorem}[section]
\newtheorem{lemma}[theorem]{Lemma}
\newtheorem{problem}[theorem]{Problem}

\newtheorem{corollary}{Corollary}[theorem]
\newtheorem{proposition}{Proposition}[theorem]
\newtheorem{claim}[theorem]{Claim}

\newtheorem{fact}[theorem]{Fact}
\newcommand{\ignore}[1]{}
\newtheorem*{acknowledgements}{Acknowledgements}

\title{Towards 
Trans-Exponential O-minimal Expansion of $(\mathbb{R},+,\cdot, 0, 1 <)$}
\author{Yayi Fu}
\date{}
\begin{document}
\maketitle
\begin{abstract}
We add an analytic trans-exponential function $\varphi$ to $\mathbb{R}_{an,\exp}$.
    We reduce the o-minimality of $\mathbb{R}_{an,\exp,\varphi}$
    to  the existence of ``many" regular values for some definable systems of functions,
    which is  a necessary condition for  the  o-minimality of $\mathbb{R}_{an,\exp,\varphi}$.
\end{abstract}

\section{Introduction}\label{intro}
\indent

An \textit{o-minimal} expansion of the  ordered field $(\mathbb{R},+,\cdot, 0, 1 <)$ is an 
expansion $(\mathbb{R},+,\cdot,0,1,<,...)$ such that every one-variable definable set is a finite union of points and intervals.
Applications of o-minimal theory to geometry gave motivation to study what expansion of the  ordered field $\overline{\mathbb{R}}$ is o-minimal. 
It's well-known that  the  ordered field $\overline{\mathbb{R}}$ itself (Tarski-Seidenberg), 
$\mathbb{R}_{\exp}$, $\mathbb{R}_{\log}$ (with $e^x$ or $\log(x)$ 
added \cite{wilkie1999theorem}), $\mathbb{R}_{\text{an, $\exp$}}$
(with $\exp$ and all analytic functions restricted on compact intervals added \cite{van1994elementary}) are o-minimal. 
But none of these known structures has a definable trans-exponential function. 
\footnote{A \textit{trans-exponential} function is a function $T:\mathbb{R}\to\mathbb{R}$ where for each $i\in\mathbb{N}$, there is  $N_i\in\mathbb{N}$ such that for all $x\geq N_i$, $T(x)>\exp_i(x)$. 
Here, $\exp_i$ means $\exp$ composed for $i$ times.
In other words, a trans-exponential function is a function that eventually grows faster than any finite compositions of $\exp$ at infinity.
Notice that we want a function that exceeds $\exp_i$ at infinity, not at some finite point on the real line.
Consider the function $\dfrac{1}{x-1}$.
Certainly $\dfrac{1}{x-1}$ eventually exceeds $\exp_i$ for all $i\in\mathbb{N}$ as $x$ approaches to $1$, but this is not what model-theorists are looking for.} 
A natural question would be to ask whether it is possible to find an o-minimal expansion of the real ordered field in which there exists a definable trans-exponential function.

In this paper, we will use  Wilkie's o-minimal test \cite[Theorem~1.9]{wilkie1999theorem},
an analytic function satisfying some Abel equation,  Milnor's method in \cite{milnor1964betti}
to reduce the existence of a trans-exponential o-minimal expansion of $\overline{\mathbb{R}}$ to one of its necessary condition.

For basic facts about o-minimality, see e.g. \cite{van1998tame} and\cite{coste2000introduction}.

Here's an outline of the paper:
The only new thing is Section \ref{prop}. Other sections are just rephrasing existing tools for expository purposes.

Section \ref{intro} gives basic definitions and the  existence of a trans-exponential function.

In Section \ref{prop}, we will  bound the number of non-singular zeroes for  functions of a certain form definable in  $\mathbb{R}_{an,\exp,\varphi}$ (Theorem \ref{rtb}), under some regularity assumption. 

In Section \ref{proof}, we will use Wilkie's o-minimal test (Fact \ref{omintest}) to show that under the  regularity assumption in Theorem \ref{rtb},
  $\mathbb{R}_{an,\exp,\varphi}$ is o-minimal,
by bounding $\gamma(A)$
(Definition \ref{gammaA}) for any quantifier-free definable $A$ in $\mathbb{R}_{an,\exp,\varphi}$.
As in \cite{milnor1964betti} and \cite{khovanskiui1991fewnomials}, 
we will reduce the problem of bounding the number of connected components to the problem of bounding the number of non-singular zeroes of a system of equations.
Consequently, we have the main theorem (Theorem \ref{mainthm}):
 \begin{theorem}
 Suppose that $\varphi$ is an analytic trans-exponential function satisfying the regularity assumption in Theorem \ref{rtb}.

 Then $\mathbb{R}_{an,\exp,\varphi}$ is o-minimal and trans-exponential.
 \end{theorem}
\begin{acknowledgements}
   \ignore{ The author would like to thank Andrew Putman for his first-year algebraic topology course in University of Notre Dame, without which the author would not be able to understand Milnor's works.
}
    The author would like to thank Lou van den Dries for comments and questions.
\end{acknowledgements}

\subsection{Definitions}
\indent
\begin{notation}    
    We use the following notations in this paper.
    \begin{itemize}
        \item 
   $\overline{\mathbb{R}}$ denotes the  ordered field
     $(\mathbb{R},+,\cdot, 0, 1 <)$.
     \item Given a function $\varphi$,
     $\mathbb{R}_{an,\exp,\varphi}$ denotes the expansion
     of $\mathbb{R}_{an,\exp}$ with $\varphi$ added.
     \item $\exp_n$, $\log_n$ denote $\exp$, $\log$ composed for $n$ many times resp..
     \item $S^n$ denotes the $n$-dimensional sphere.
       \item $\underset{\longrightarrow}{\lim}$ denotes the direct limit.
         \item $B^D_{<r}(x)$ denotes the intersection of $D$ with the ball centered at $x$ with radius $r$. 
         When $D$ is $\mathbb{R}^n$, we just write  $B_{<r}(x)$.
         $B^D_{\leq r}(x)$ is defined similarly.
         \item $Gr(k,n)$ denotes the Grassmannian of $k$-dimensional subspaces of $\mathbb{R}^n$.
      \item $\check{H}^*(V;G)$ denotes the C$\check{\text{e}}$ch cohomology group of $V$ with coefficient $G$.
      \item ${H}_*(V;{G})$ denotes the singular homology group of $V$ with coefficient $G$.
      \item ${H}^*(V;{G})$ denotes the singular cohomology group of $V$ with coefficient $G$.
         \item Given a $C^1$ function $P$ and given $y$ in the domain of $P$, $J_{{y}}( {P})$ denotes the Jacobian matrix of $P$ at $y$.
        \item For $\overline{x}=(x_1,\ldots,x_n)$, $\overline{x}^2$ denotes
        $x_1^2+\cdots+x_n^2$.
      \end{itemize}
\end{notation}
\ignore{
Given $\varphi:I\to \mathbb{R}$ a smooth function where $I$ is an interval in $\mathbb{R}$ and the expansion $\mathbb{R}_{an,\exp,\varphi}$,
\textit{definability} means definablility in the model $\mathbb{R}_{an,\exp,\varphi}$ with parameters.}
\begin{definition}
     Given an expansion $\mathcal{R}$ of $\overline{\mathbb{R}}$,  a \textit{term-definable} function in $\mathcal{R}$ is a  function that can be represented as a term in $\mathcal{R}$.
    \footnote{We add the name \textit{term-definable} 
       in order to distinguish a term from the function it defines.
       It's also important to notice that not every quantifier-free definable function is term-definable.
       For example, $\dfrac{1}{x}$ is quantifier-free definable in $\overline{\mathbb{R}}$, but it is not term-definable since terms in $\overline{\mathbb{R}}$ are polynomials.}
      For example, the family of term-definable functions in $\mathbb{R}_{an,\exp,\varphi}$ consists of definable functions constructed as follows:
       \begin{itemize}
           \item The  term-definable functions in $\mathbb{R}_{an,\exp}$.
           \item If a function $f$ is term-definable, then $\varphi\circ f$ is term-definable. 
           \item If   $\varphi_1,\ldots,\varphi_n$ are term-definable functions, and $P$ is term-definable in $\mathbb{R}_{an,\exp}$,
           then $P(\overline{X}, \varphi_1,\ldots,\varphi_n)$ is term-definable. 
       \end{itemize}

       \end{definition}
 
       \begin{definition}
Let $\mathcal{F}$ be the smallest family satisfying the following:
\begin{itemize}
    \item If $f$ is a function definable in $\mathbb{R}_{an,\exp}$ 
    with domain $\mathbb{R}^m$, then  $f\in\mathcal{F}$.
    \item  If $g_1,...,g_n\in\mathcal{F}$ and 
    $P$ is definable in $\mathbb{R}_{an,\exp}$ with domain $\mathbb{R}^m$,
    then $P(\overline{X},g_1(\overline{X}),...,g_n(\overline{X}))\in\mathcal{F}$. 
    \item If $g\in \mathcal{F}$, then $\varphi\circ g\in\mathcal{F}$ and  $\varphi'\circ g\in\mathcal{F}$. 
\end{itemize}

 Given a function $f$ in $\mathcal{F}$, we define the \textit{formal complexity} of $f$, denoted by \textit{$fcpx(f)$} as follows:
 \begin{itemize}
     \item If $f$ is a function definable in $\mathbb{R}_{an,\exp}$ with domain $\mathbb{R}^m$, then $fcpx(f)=0$. 
      \item If $f=P(\overline{X},g_1(\overline{X}),...,g_n(\overline{X}))$,
      where $P$ is definable in $\mathbb{R}_{an,\exp}$ with domain $\mathbb{R}^m$,
      then $fcpx(f)=\max\{fcpx(g_1),\ldots,fcpx(g_n)\}$. 
     \item If $f$ has the form $\varphi\circ g$ or $\varphi'\circ g$ where $fcpx(g)=k$, then $fcpx(f)=k+1$.
 \end{itemize}
\end{definition}

\begin{definition}
    A \textit{$\varphi$-monomial} (resp. \textit{$\varphi'$-monomial}) is $\varphi\circ f$
    (resp.  $\varphi'\circ f$)
    where $f$ is a term-definable function in $\mathbb{R}_{an,\exp,\varphi}$.
    
\end{definition}

\subsection{Existence of a Trans-Exponential Function}
\begin{fact}\cite[Satz~9]{kneser1950reelle}
\label{trexpf}
    There is a strictly increasing analytic solution $\psi:\mathbb{R}\to\mathbb{R}$ satisfying the Abel equation 
    \[
    \psi(e^x)=\psi(x)+1.
    \]
\end{fact}
 (See also \cite[Theorem~4.2]{bonet2015abel} for the existence of monotonic analytic solutions to a general Abel equation.)

 The following properties were stated in \cite{kneser1950reelle}.
 We write the details here for completeness.
 \begin{proposition}\label{addpsi}
Fix $\psi$ a strictly increasing analytic solution to the Abel equation given in Fact \ref{trexpf}.

Then for this $\psi$,
    \begin{enumerate}
        \item for all $n\in\mathbb{N}$, $x\in\mathbb{R}$, 
        $\psi(x)=\psi(\exp_n(x))-n$;\\
         for all $n\in\mathbb{N}$, $x>\exp_n(1)$,   
        $\psi(\log_n(x))+n=\psi(x)$;
        \item for all $n\in\mathbb{N}$,
        $x\in\mathbb{R}$, 
        $\psi'(x)=
        \psi'(\exp_n(x))\cdot \exp_n^{\prime} (x)$;\\
          for all $n\in\mathbb{N}$, $x>\exp_n(1)$,   
        $\psi'(\log_n(x))\log_n'(x)=\psi'(x)$;
        \item for all  $n\in\mathbb{N}$, if
        $x>\exp_n(1)$, then  
        $  |\psi'(x)|\leq \underset{y\in[1,e]}{\sup}|\psi'(y)|
   \cdot \log_n'(x)
   $;
        \item for all  $n\in\mathbb{N}$, if
        $x>\exp_n(1)$, then  
        $  |\psi(x)|\leq
      \sup_{y\in[1,e]}|\psi'(y)|
   \cdot 
 \left| \log_n (x)-
 1\right|$;\\ 
        hence for all $n\in\mathbb{N}$, $|\psi|$ is eventually dominated by $\log_n$.
        \item $\underset{x\to+\infty}{\lim}\psi(x)=+\infty$;
         $\underset{x\to+\infty}{\lim}\psi'(x)=0$;
         $\psi(x)$ is bounded on $(-\infty,0]$.
    \end{enumerate}
\end{proposition}
\begin{proof}
  1. is straightforward by the Abel equation $ \psi(e^x)=\psi(x)+1$ and induction on $n$.
  2. follows from 1. together with Chain Rule.
Notice that we assumed $x>\exp_n(1)$ when we mentioned $\log_n$ because $\log_n$ is not defined on the whole interval $(0,+\infty)$.
Also, because $\log_n$  is a composition of  increasing functions, $\log_n$ is increasing and $\log_n'(x)\geq 0$ whenever it is defined.
 
For 3., fix $n\in\mathbb{N}$.
If $x>\exp_n(1)$, then  
   $x\in (\exp_m(1), \exp_{m+1}(1)]$ for some $m\geq n$, and hence
   \begin{align}\label{estderi}
   |\psi'(x)|
   &=|\psi'(\underbrace{\log_m(x)}_{\text{$\in (1,e]$}})
   \cdot \log_m^{\prime} (x)|\quad\text{, by 2.}
   \notag
   \\
   &\leq
   \sup_{y\in[1,e]}|\psi'(y)|
   \cdot 
   \log_m^{\prime} (x)
   \notag
   \\
    &\leq
   \sup_{y\in[1,e]}|\psi'(y)|
   \cdot \log_n^{\prime} (x).
   \end{align}
   The last step is because $\log'(y)<1$ on $(1,+\infty)$ and $x\in (\exp_m(1), \exp_{m+1}(1)]$ for some $m\geq n$.

  For 4., given $n\in\mathbb{N}$, and $x>\exp_n(1)$, 
   \begin{align}\label{logbd}
   &|\psi(x)-
   \psi( \exp_n(1))
   |
   \notag\\
   =&\left|\int_{ \exp_n(1)}^{x}\psi'(y)\right|
   \quad\text{, by FTC}
   \notag\\
   =&\int_{\exp_n(1)}^{x}
   \psi'(y)
 \quad  \text{, $\psi'\geq 0$ on $(0,+\infty)$}
 \notag\\
   \leq &
    \sup_{y\in[1,e]}|\psi'(y)|
   \cdot 
   \int_{ \exp_n(1)}^{x}
 \log_n^{\prime} (y)
 \quad\text{, by (\ref{estderi})}
 \notag\\
   =& \sup_{y\in[1,e]}|\psi'(y)|
   \cdot 
 \left(  \log_n (x)-
 \log_n (\exp_n(1))\right)
 \quad
 \text{, by FTC}
 \notag \\
  =& \sup_{y\in[1,e]}|\psi'(y)|
   \cdot 
 \left( \log_n (x)-
 1\right).
   \end{align}
   Given a fixed $n\in\mathbb{N}$,
   choose $m\in\mathbb{N}$ with $m>n$ such that for all $x>\exp_m (1)$,
   \[
   \sup_{y\in[1,e]}|\psi'(y)|
   \cdot 
 \left(  \log_m (x)-
 1\right)
 +|\psi(\exp_m(1))|
 \leq \log_n (x).
   \]
   This is possible because $\psi(\exp_m(1))=\psi(1)+m$, 
   $\dfrac{1}{2}\log_n(x)>\dfrac{1}{2}\exp_{m-n}(1)$ for all $x>\exp_m(1)$,
   and $\dfrac{1}{2}\exp_{m-n}(1)$ grows faster than $\psi(1)+m$ as $m$ grows.
   Then for all $x>\exp_m (1)>\exp_n (1)$,
   \begin{align*}
       |\psi(x)|&\leq \sup_{y\in[1,e]}|\psi'(y)|
   \cdot 
 \left(  \log_m (x)-
 1\right)
 +|\psi(\exp_m(1))|
 \quad\text{, by (\ref{logbd})}\\
 &\leq \log_n (x)\quad\text{, by choice of $m$}.
   \end{align*}
   It follows that for a fixed $n\in\mathbb{N}$,
   $|\psi(x)|\leq  \log_n(x) $ whenever $x$ is large.
   
   For 5.,
   $\underset{x\to+\infty}{\lim}\psi(x)=+\infty$ 
   because $\psi$ is monotonic and
   \[\psi(\exp_n(1))=\psi(1)+n\to +\infty\]
   as $n\to+\infty$.
   
          $\underset{x\to+\infty}{\lim}\psi'(x)=0$ is obvious from 3.
          
        For $x\in(-\infty,0]$,
        $\exp(x)\in(0,1]$.
    So for all $x\in(-\infty,0]$,
    $\psi(x)\in (\psi(0)-1,\psi(1)-1]$.
        Hence $\psi(x)$ is bounded on $(-\infty,0]$.
\end{proof}

We fix the symbol $\varphi$ for this chosen function.
By induction, we have the following simple property:
\begin{lemma}
  \label{mfprop}
  Given  a term-definable smooth function $F$ in $\mathbb{R}_{an,\exp,\varphi}$ with domain $\mathbb{R}^n$,
  there exists $s \in\mathbb{R}$  such that for all $x\in\mathbb{R}^n$,  \[
 |F(x)|\leq\exp_s(\|x\|). 
  \]
\end{lemma}
\newpage

\section{Bounding \# of Solutions }\label{prop}
\indent

The goal in this section is to prove Theorem \ref{rtb}.

\subsection{O-minimal Preliminaries}
\begin{definition}
Let $\mathcal{M}$ be an elementary extension of $\mathbb{R}_{an,\exp}$.

       \begin{enumerate}    \item[(i)]\cite[Definition~1.1]{pillay1988groups}
       Given \(B \subseteq \mathcal{M}\) and  \(\bar{a} \in \mathcal{M}^n\), 
        \(\dim_{alg}(\bar{a}/B)\) is defined as the 
        {least cardinality} of a subtuple \(\bar{a}'\) of \(\bar{a}\) such that \(\bar{a} \subseteq \text{dcl}(B \cup \bar{a}')\).
       \item[(ii)] \cite[Definition~1.3]{pillay1988groups}
       Given \(B \subseteq \mathcal{M}\) and a \(B\)-definable
       \(X \subseteq \mathcal{M}^n\),
       \(
\dim_{alg}^{\mathcal{M}} X \) is defined as \(\max\{\dim_{alg}(\bar{a}/B): \bar{a}\in X\}
\).
Given $\bar{a}\in X$,
\(\bar{a}\) is a \textit{generic point of \(X\) over \(B\)} if \(\dim_{alg}(\bar{a}/B) = \dim_{alg}^{\mathcal{M}} X\).
\item[(iii)]  
Given \(B \subseteq \mathcal{M}\) and a \(B\)-definable
       \(X \subseteq \mathcal{M}^n\),  \(
\dim_{geo}^{\mathcal{M}} X \) is defined as the maximal $k$ such that  some
projection of $X$ onto $\mathcal{M}^k$ has non-empty interior in $\mathcal{M}^k$
    \end{enumerate}      
     \end{definition}
     Observe that for all elementary extensions   $\mathcal{M}\prec\mathcal{N}$ of $\mathbb{R}_{an,\exp}$,
     and all $X$ definable in $\mathcal{M}$,
     $\dim_{geo}^\mathcal{N}(X)=\dim_{geo}^{\mathcal{M}}(X)$. 
     For example, fix a cell-decomposition of $X$ and for each cell, fix a definable homeomorphism between the cell and $\mathcal{M}^k$ for some $k$.
     Then in $\mathcal{N}$, these sets form a cell decomposition of $X$ in $\mathcal{N}$ and the same homeomorphisms 
     are now between the cells and $\mathcal{N}^k$.
     We use $\dim_{geo}(X)$ to denote  this fixed number.
     
     \begin{fact}
         \cite[Lemma~1.4]{pillay1988groups}
         
         In an $\omega_1$-saturated extension  
         $\mathcal{M}$ of $\mathbb{R}_{an,\exp}$,
        $\dim_{alg}^{\mathcal{M}}(X)=\dim_{geo}(X)$.
     \end{fact}
   \begin{definition}
       \cite[Definition~2.4]{peterzil2008complex}
Let $f:X\to Y$ be a function where $X \subseteq \mathbb{R}^n$ and $Y \subseteq \mathbb{R}^k$.

Given $b\in Y$, we say that $f$ is \textit{bounded over $b$} if, in the Euclidean topology,
there is a neighborhood $W \subseteq Y$ of $b$ such that $f^{-1}(W)$ is a bounded subset of $\mathbb{R}^n$.
   
    Given $k,n\in\mathbb{N}$ with $k\leq n$,
    fix a  parametrization of $Gr(k,n)$ definable in $\mathbb{R}_{an,\exp}$.
    In an $\omega_1$-saturated extension of $\mathbb{R}_{an,\exp}$,
    given $B\subseteq \mathcal{M}$ and a $k$-subspace $L$,
    we say that \textit{$L$ is generic over $B$} if the 
    parameter representing $L$ in $Gr(k,n)$ is a generic point of $Gr(k,n)$ over $B$.
    \end{definition}
\begin{fact}\label{projprop}
    \cite[Lemma~2.13]{peterzil2008complex}
    Let $\mathcal{M}$ be an $\omega_1$-saturated extension of $\mathbb{R}_{an,\exp}$.
    Let $S \subseteq \mathcal{M}^n$ be a  locally closed set which is  
    $B$-definable in $\mathcal{M}$ and
     $\dim_{geo}(S)=k < n$.
    Let $L$ be a  $k$-dimensional $\mathcal{M}$-subspace of $\mathcal{M}^n$ generic over $B$.
    Let $\pi : \mathcal{M}^n \to L$ be the orthogonal projection. 
    Then for all $y \in L$, $\pi|_A$ is bounded over $y$.

\end{fact}
\begin{proof}
    Repeat the proof in  \cite{peterzil2008complex}. 
\end{proof}
\begin{corollary}\label{gengr}
     Let $S \subseteq \mathbb{R}^n$ be a  locally closed set which is  
    $B$-definable in $\mathbb{R}_{an,\exp}$ and $\dim_{geo}(S)=k < n$.
    Then for the set 
    \[H:=\{L\in Gr(k,n): \forall y\in L\, \,\pi_L:S\to L\text{ is bounded over $y$}\},
    \]
    $\dim_{geo}(Gr(k,n)\setminus H)<\dim_{geo}(Gr(k,n))$.
\end{corollary}
\begin{proof}
    It follows from Fact \ref{projprop} and \cite[Lemma~1.12]{pillay1986some}.
\end{proof}

\begin{corollary}\label{gengral}
     Let $S \subseteq \mathbb{R}^n$ be a  locally closed set which is  
    $B$-definable in $\mathbb{R}_{an,\exp}$ and $\dim_{geo}(S)=k < n$.
Let $H$ be as in Corollary \ref{gengr}.
    
    Then for the set 
    \[T:=\{A\in GL_n(\mathbb{R})
    :  A(\mathbb{R}^k)
    \in H\},
    \]
    $T$ is dense in $GL_n(\mathbb{R})$.
\end{corollary}
\begin{proof}
(See e.g. \cite[Example~21.21, Example~1.36]{lee2012smooth} for the topological structure of $Gr(k,n)$ that makes it a homogeneous space.)

Consider the map $proj:GL_n(\mathbb{R})\to Gr(k,n)$ defined by $proj (A)=A(\mathbb{R}^k)$.
$proj$ is open, so the preimage of a dense subset of $ Gr(k,n)$ is dense in $GL_n(\mathbb{R})$.
By definition, $T=proj^{-1}(H)$,
 and by Corollary \ref{gengr},
 $H$ is dense in $Gr(k,n)$.
 Hence $T$ is dense in $GL_n(\mathbb{R})$.
 
\end{proof}

 \begin{lemma}\label{proppro}
     Let $S\subseteq\mathbb{R}^n$ be definable in 
     $\mathbb{R}_{an,\exp}$ and $L$ be a $k$-subspace of $\mathbb{R}^n$.
     Suppose that for all $w\in L$, the orthogonal projection $\pi_L:S\to L$ is bounded over $w$.
     Then for all $M\in\mathbb{R}$,
     $\sup\{\|y\|:\exists w\,\,(y,w)\in S,
     w\in L, \|w\|\leq M\}$ is a real number.
 \end{lemma}
 \begin{proof}
 Fix $M\in \mathbb{R}$.
For each $w\in L$ with $\|w\|\leq M$, since $\pi_L$ is bounded over $w$,
we can choose $U_w$  a neighborhood of $w$ such that $ \pi_L^{-1}(U_w)$ is a subset of $\mathbb{R}^n$ bounded by some $N_w\in\mathbb{R}$.
Since $\{w\in L:\|w\|\leq M\}$ is compact, there exist $w_1,...,w_t\in L$ such that \[
\{w\in L:\|w\|\leq M\}\subseteq U_{w_1}\cup\cdots\cup U_{w_t}.
\]
It follows that 
\[
\sup\{\|y\|:\exists w\,\,(y,w)\in S,w\in L, \|w\|\leq M\}\leq \max\{N_{w_1},...,N_{w_t}\}
\]is a real number.
 \end{proof}
 In the following lemma, given a definable cell, 
 {$\partial C$ denotes the boundary of $\overline{C}$ as a manifold with boundary, not the usual topological boundary.
 For example, in $\mathbb{R}^2$, the topological boundary of $C:=\{(t,0)
 \in \mathbb{R}^2:
 t\in(0,1)\}$} is the whole $\overline{C}$, but $\partial C$ is $\{(0,0), (1,0)\}$.
 
 \begin{lemma}\label{domfun}

Suppose that $F:(0,+\infty)\times
         S\to\mathbb{R}^{\geq 0}$ is  a definable  (not necessarily continuous) function in $\mathbb{R}_{an,\exp}$,
    where for each $t\in S$, $g_t:=F(-,t)$ is an increasing  (not necessarily continuous) function
from $(0, +\infty)$ to $\mathbb{R}^{\geq 0}$.

     Then there exist
     \begin{itemize}
         \item 
         an increasing definable (not necessarily continuous) function $g:(0,+\infty)\to \mathbb{R}^{\geq 0}$ ,
         \item a definable  (not necessarily continuous) function $h:  S\to \mathbb{R}^{\geq 0}$
     \end{itemize}
     such that for all $t\in S$
     and all $x$ with $x\geq h(t)$,
     $g(x)\geq g_t(x)$.
     i.e. $g$ eventually dominates all $g_t$ with $t\in S$.

     \end{lemma}
  \begin{proof}
Let  $C_1\cup\cdots\cup C_m$ be a cell decomposition of $(0,+\infty)\times S$ 
such that for each $i\in [m]$, $F$ is continuous on $C_i$.
Since we want to bound the $g_t$'s eventually,
we may assume that for all $i\in[m]$  there is a continuous function $\kappa_i:\pi_S(C_i)\to \mathbb{R}$
such that for all $x\in (0,+\infty)$, 
if $x\geq \kappa_i(t)$, then $(x,t)\in C_i$.
Also, by definition of cells, we may assume that for each $i\in[m]$, $\pi_S(C_i)$ is a cell.

For each $(i,x)\in[m]\times(0,+\infty)$, define  
\[
D_{i,x}:=
\{t\in \pi_S(C_i):
d(t, \partial  \pi_S(C_i))
\geq \frac{1}{x}, \|t\|\leq x\}.
\]
Each $D_{i,x}$ is equal to $\{t\in \overline{\pi_S(C_i)}:
d(t, \partial  \pi_S(C_i))
\geq \frac{1}{x}, \|t\|\leq x\}$, and hence is compact.

Define a function $g:(0,+\infty)\to \mathbb{R}$ by:
\[
g(x)=
 \sup \{F(x,t):t\in \bigcup_{i\in[m]}(D_{i,x}\cap\{t\in\pi_S(C_i):x\geq  \kappa_i(t)\}).
 \]
By continuity of $F$ on each $C_i$,
$ \sup \{F(x,t):t\in \bigcup_{i\in[m]}(D_{i,x}\cap\{t\in\pi_S(C_i):x\geq  \kappa_i(t)\})$ is a real number.
$g$ is increasing because for all $x<y$, all $i\in [m]$,
$D_{i,x}\subseteq D_{i,y}$,
and because each $g_t$ is increasing.

Define $h:S\to\mathbb{R}^{\geq 0}$ by \[
h(t)=
\max\{\|t\|,
 \max_i\{\kappa_i(t):t\in\pi_S(C_i)\},
\frac{1}{{\min_i}\{
d(t, \partial  \pi_S(C_i)):
t\in  \pi_S(C_i)\}}
\}.
\]
Notice that if $t\in \pi_S(C_i)$, then $d(t, \partial\pi_S(C_i))\neq 0$ since $\pi_S(C_i)$ is a cell.

Now we check that $g(x)\geq g_t(x)$ for all $t\in S$ and all $x$ with $x\geq h(t)$.
Fix $t_0\in S$ and  $x$ with $x\geq h(t_0)$.
Since \[
x\geq \max\{\|t_0\|,
 \max_i\{\kappa_i(t_0):t_0\in\pi_S(C_i)\},
\frac{1}{{\min_i}\{
d(t_0, \partial  \pi_S(C_i)):
t_0\in  \pi_S(C_i)\}}
\},
\]
we have that
 $t_0\in \underset{i\in[m]}{\bigcup}(D_{i,x}\cap\{t\in\pi_S(C_i):x\geq  \kappa_i(t)\})$.
By the definition of $g$,
$g(x)\geq F(x,t)=g_t(x)$.
     \end{proof}
     
 \subsection{The Equation System and Related Sets}
 \indent 
 
We will consider systems of the following form.
 \begin{equation}
 \begin{cases}\label{gensys}
     P_1(X_1,...,X_n, l_1,..,l_{n+1},\epsilon_1,...,\epsilon_n, \delta,
     \varphi_1,...,\varphi_k,\psi_1,...,\psi_u
     )=0
     \\...\\...\\...\\
P_n(X_1,...,X_n,l_1,..,l_{n+1},\epsilon_1,...,\epsilon_n,\delta,
     \varphi
     _1,...,\varphi_k,
\psi_1,...,\psi_u)=0,
 \end{cases}
 \end{equation}
 where
 \begin{itemize}
     \item $P_1,...,P_n$ are  smooth function with domain $\mathbb{R}^{3n+k+u+2}$ definable in $\mathbb{R}_{an,\exp}$,
     \item $(\bar{l},\bar{\epsilon},\delta)\in[-1,1]^{2n+2}$ are constants,
 \item $\varphi_i$'s, $\psi_i$'s are $\varphi$-monomials, $\varphi'$-monomials resp..
 \end{itemize}

Given a function $P(\overline{X},\overline{W})=(P_1(\overline{X},\overline{W}),...,P_n(\overline{X},\overline{W}))$ definable in $\mathbb{R}_{an,\exp}$ as in (\ref{gensys}),
where $\overline{W}$ are variables that substitute for $(\bar{l},\bar{\epsilon},\delta,\overline{\varphi},\overline{\psi})$,
define
\begin{align}\label{gset}
G=\{(A,\overline{\eta})
    \in
    GL_m(\mathbb{R})\times[-1,1]^{n}:
    \forall w\in \mathbb{R}^{2n+k+u+2} \,\,\notag\\
    \pi_w:Z(P\circ A-\overline{\eta})\to \mathbb{R}^{2n+k+u+2}\text{ is bounded over $w$}\},
\end{align}
 where $\pi_w$ is the projection onto the $w$-coordinate, and $m=3n+k+u+2$;
 
define $F:(0,+\infty)\times G\to\mathbb{R}^{\geq 0}$ by
\begin{align}\label{fmap}
   F(x,A,\overline{\eta})=\sup\{\|y\|:\exists w\,\,\|w\|\leq x\text{ and } 
   (y,w)\in Z(P\circ A-\overline{\eta})
   \};
\end{align}
and define
 \begin{align}\label{rset}
R&=\{(A,\overline{\eta},\overline{l},\overline{\epsilon},\delta)
    \in
 G\times[-1,1]^{2n+2}:\notag
 \\
&P\circ A(\overline{X},\overline{l},
\overline{\epsilon},
\delta,
\overline{\varphi},
\overline{\psi}
)
-\overline{\eta}
 \text{ has non-singular zeros only}\}.
\end{align}
{This $R$ is definable in $\mathbb{R}_{an,\exp, \varphi}$.

Notice that in (\ref{fmap}), $F$ is well-defined,
since for each $(x,A,\overline{\eta})\in (0,+\infty)\times G$,
by Lemma \ref{proppro}, $F(x,A,\overline{\eta})$ is a real number.

\begin{outline}
 Given a system of the form 
 in (\ref{gensys}),
if for  $P=(P_1,...,P_n)$, the function \[ 
g(x)=\sup\{\|y\|:\exists w\,\|w\|\leq x, (y,w)=p \text{ for some }p\in Z(P) \}
\]
is bounded by $\exp_s$
 for some $s$,
 then  because $|\varphi|$ is eventually dominated by all $\log_n$,
 the zeroes of $P$ are contained in a compact set and 
 hence the  $\varphi$-monomials are reduced to lower complexity.
 For an arbitrary $P$, there's no guarantee for this property.
 However, we can tilt the space a little so that  $P\circ A$ in place of $P$ has this property,
 where $A$ is a change of coordinate from the new coordinate to the original one.
 Moreover, by o-minimality of $\mathbb{R}_{an,\exp}$ , we can choose a path $\gamma:[0,1]\to GL_m(\mathbb{R})\times [-1,1]^{n}$, 
 $\gamma(t)=(A(t),\overline{\eta}(t))$,
 so that $A(0)=id$, $\overline{\eta}(0)=0$,
 and for all $t\in(0,1]$, $P\circ A(t)=\overline{\eta}(t)$ 
 has the boundedness property described above.
 So for each $t\in(0,1]$, there is a neighborhood $U$ of $t$ such that 
 the zeroes of (\ref{gensys}) with $P=0$ replaced by $P\circ A(t)=\overline{\eta}(t)$  are contained in a compact set.

 Using a connectedness argument, by Implicit Function Theorem,
 if there is a triangulation for 
$R$ defined in (\ref{rset})
 that is locally finite in $\overline{G}\times[-1,1]^{2n+2}$,
 then
we can show that the number of non-singular zeroes in (\ref{gensys}) is bounded by the number of non-singular zeroes in 
  (\ref{gensys}) with $P$ replaced by $P\circ A(c)$ for some $c>0$.
Moreover, our argument will show that this bound is independent of the parameters $l_1,..,l_n,\epsilon_1,...,\epsilon_n, ,\delta$.
 \end{outline}
We show some denseness property about  $G,R$ defined in (\ref{gset}), (\ref{rset}) resp..

\begin{lemma}\label{bddden}
    Given a function $P$ of the form in (\ref{gensys}), let $G$ be the set defined in (\ref{gset}).
    Let $m=3n+k+u+2$.
    
Then for the set   \[
S:=\{A\in GL_m(\mathbb{R}):
\dim_{geo}([-1,1]^{n}\setminus G_A)
= n\},
\]
$\dim_{geo}(S)<\dim_{geo}(GL_m(\mathbb{R}))$.
Here, $ G_A$ is the fiber $\{\overline{\eta}\in[-1,1]^{n}:(A,\overline{\eta})\in G\}$.
\\

In particular, $G$ is dense in $GL_m(\mathbb{R})\times[-1,1]^{n}$.
\end{lemma}
\begin{proof}
If $\dim_{geo}(S)=\dim_{geo}(GL_m(\mathbb{R}))$, then   \[
\dim_{geo}(
\bigcup_{A\in S}[-1,1]^{n}\setminus G_A)
=\dim_{geo}(GL_m(\mathbb{R}))+n.
\]  
So there is $\overline{\eta}$ such that for the fiber $G_{\overline{\eta}}=\{
A\in GL_m(\mathbb{R}):(A,\overline{\eta})\in G\}$, $\dim_{geo}(GL_m(\mathbb{R})
\setminus G_{\overline{\eta}})=
\dim_{geo}(GL_m(\mathbb{R}))
$.
But this is impossible because by Corollary \ref{gengral},
for each $\overline{\eta}$, the fiber $G_{\overline{\eta}}$ is dense in $GL_m(\mathbb{R})$.

\end{proof}

\begin{lemma}\label{subandenreg}
Given a system $P$ of the form (\ref{gensys}), let $G,R$ be defined in (\ref{gset}), (\ref{rset}) resp..
Let $m=3n+k+u+2$.

Then $R$  is dense in $G\times[-1,1]^{2n+2}$.

\end{lemma}
\begin{proof}

Fix $(A,\overline{\eta},\overline{l},\overline{\epsilon},\delta)
    \in
 G\times[-1,1]^{2n+2}$ and $r>0$.
 By Lemma \ref{bddden}, there is $A'\in GL_m(\mathbb{R})$ such that  $\|A-A'\|<\dfrac{r}{2}$ and
 $\dim_{geo}([-1,1]^{n}\setminus G_{A'})
< n$.
By Sard's theorem, the set of regular values for $P\circ A'(\overline{X},\overline{l},
\overline{\epsilon},
\delta,
\overline{\varphi},
\overline{\psi}
)$
is dense in $[-1,1]^n$.
Hence,
 there exists
 $\overline{\eta}'\in[-1,1]^{n}$ such that $\|\overline{\eta}-\overline{\eta}'\|<\dfrac{r}{2}$,
 $(A',\overline{\eta}')\in G$,
 and 
 that $\overline{\eta}'$ is a regular value of $P\circ A'(\overline{X},\overline{l},
\overline{\epsilon},
\delta,
\overline{\varphi},
\overline{\psi}
)$.
So given any arbitrary $(A,\overline{\eta},\overline{l},\overline{\epsilon},\delta)
    \in
 G\times[-1,1]^{2n+2}$ and $r>0$,
  we can find some $(A',\overline{\eta}',\overline{l},\overline{\epsilon},\delta)
    \in R$ whose distance from $(A,\overline{\eta},\overline{l},\overline{\epsilon},\delta)
    $ is $<r$. 
    Hence $R$ is dense  in $G\times[-1,1]^{2n+2}$.

\end{proof}

\subsection{Finite \# of Non-singular Solutions}
\indent

Given $p\in \mathbb{R}^{n+k}$, and a $k$-dimensional subspace $L$, 
$p= p_{L^{\perp}}\oplus p_L$ where $p_L$ is the orthogonal projection of $p$ on $L$.
Let $[p]_{L}$ denote the coordinate representation of $p$ wtih respect to $L$.

Given a system $P(X_1,...,X_n,  W_1,...,W_k)=0$, and given a $k$-dimensional subspace $L$ which will correspond to $W_1,...,W_k$,
$p\in Z(P)$ iff
$P(A_L([p]_L))
=P\circ A_L([p]_L)=0$,
where $A_L$ is the change of coordinate from $L^\perp\oplus L$ to the original one.
It follows that for all $w\in \mathbb{R}^k$, and all $y\in \mathbb{R}^n$,
$(y,w)=[p]_L$ for some $p\in Z(P)$ iff $(y,w)\in Z(P\circ A_L)$.
Hence, if the orthogonal projection $\pi_L$ satisfies that
\[
\forall w\in L\, \,\pi_L:Z(P)\to L\text{ is bounded over $w$},\]
then \[
 \forall w\in \mathbb{R}^{k} \,\,\notag\\
    \pi:Z(P\circ A_L)\to \mathbb{R}^{k}\text{ is bounded over $w$}.
\]

\begin{lemma}\label{indbdd}
    Given a system $P$ of the form (\ref{gensys}), let $G$ be defined as in (\ref{gset}).
    
    Then
    \begin{enumerate}
\item there is a cell decomposition $D_1\cup\cdots\cup D_q$ of $G$ in $\mathbb{R}_{an,\exp}$ such that 
given $i\in[q]$ and  $(A_0,\overline{\eta}_0)\in D_i$,
there exist an open neighborhood $U$ of $(A_0,\overline{\eta}_0)$ in $D_i$ and $r\in\mathbb{R}$ such that
for all $(\overline{l},
\overline{\epsilon},
\delta)
\in[-1,1]^{2n+2}$,
\[
\bigcup_{(A,\overline{\eta})\in U}
Z( P\circ A(\overline{X},\overline{l},
\overline{\epsilon},
\delta,
\overline{\varphi},
\overline{\psi})
-\overline{\eta})
\subseteq B_{\leq r}(\mathbf{0});
\] 
\item for each $(A,\overline{\eta})\in G$,
$Z(P\circ A(\overline{X},\overline{l},
\overline{\epsilon},\delta,
\overline{\varphi},
\overline{\psi})
-\overline{\eta})\subseteq Z(Q)$ for some function $Q$ definable in $\mathbb{R}_{an,\exp,\varphi}$ with lower $fcpx$.
\item   Suppose that $D_1\cup\cdots\cup D_q$ is a cell decomposition in 1..
Given $i\in[q]$ and $\sigma:[0,1]\to D_i\times[-1,1]^{2n+2}$ a continuous path,
there is $r\in \mathbb{R}$ such that for all $t\in[0,1]$,
\[Z(P\circ A(t)(\overline{X},
\overline{l}(t),\overline{\epsilon}(t),
\delta(t),
\overline{\varphi},
\overline{\psi}
)
-\overline{\eta}(t)
)\] is contained in the ball $B_{\leq r}(\mathbf{0})$.
    \end{enumerate}
\end{lemma}
\begin{proof}

Write each $\varphi$-monomial $\varphi_i$ as $\varphi\circ f_i$, where $f_i$ is term-definable in $\mathbb{R}_{an,\exp,\varphi}$;
write each $\varphi'$-monomial $\psi_j$ as $\varphi'\circ g_j$, where $g_j$ is term-definable in $\mathbb{R}_{an,\exp,\varphi}$.

Let $F:(0,+\infty)\times G\to\mathbb{R}^{\geq 0}$ be defined as in (\ref{fmap}) and let $g_{A,\overline{\eta}}$ be $F(-,A,\overline{\eta})$.
$F$ is definable in $\mathbb{R}_{an,\exp}$
because $P$ is definable in  $\mathbb{R}_{an,\exp}$.
Applying Lemma \ref{domfun} to $F$, we get functions $g:(0,+\infty)\to \mathbb{R}^{\geq 0}$ ,
         $h:  G\to \mathbb{R}^{\geq 0}$
  definable in $\mathbb{R}_{an,\exp}$   such that for all $(A,\overline{\eta})
     \in G$,
     for all $x$ with $x\geq h(A,\overline{\eta})$,
     $g(x)\geq g_{A,\overline{\eta}}(x)$.
Take a cell decomposition $D_1\cup\cdots\cup D_q$ of $G$ such that for all $i\in[q]$, $h$ is continuous on $D_i$.

By properties of $\varphi,\varphi'$  and o-minimality of $\mathbb{R}_{an,\exp}$,
we can fix $
d\in\mathbb{R}$ and $s\in\mathbb{N}$ such that 
\begin{enumerate}
    \item [i.]
    for all 
$i\in[k]$ and all $x\in\mathbb{R}^n$, $|f_i(x)|\leq \exp_s(\|x\|)$;
\item [ii.] for all $x\in (0,+\infty)$, $|g(x)|\leq \exp_s(x)$;
\item [iii.] for all $z$ with $z\geq d$, $k|\varphi(\exp_{2s}(z))|<\dfrac{z}{2}$;
\item [iv.] $d> 2\cdot(2n+2+u\cdot \underset{z\in\mathbb{R}}{\sup}|\varphi'(z)|)$.
(This is possible because $|\varphi'|$ is bounded on $\mathbb{R}$.)

\end{enumerate}

Given $(A_0,\overline{\eta}_0)\in G$,
we may assume that $(A_0,\overline{\eta}_0)\in D_1$.
Let $U$ be a neighborhood of $(A_0,\overline{\eta}_0)$ in $D_1$ with $\overline{U}\subseteq D_1$.
Then $\underset{(A',\overline{\eta}')\in U}{\sup}
h(A',\overline{\eta}')$  is a well-defined real number by continuity of $h$ on $\overline{U}$.
 Given $(A,\overline{\eta})\in U$   and 
 a zero  \(y\in Z(P\circ A(\overline{X},\overline{l},
\overline{\epsilon},
\delta,
\overline{\varphi},
\overline{\psi}
)-\overline{\eta})\),
 if ${w}=
(\overline{l},
\overline{\epsilon},
\delta,
\overline{\varphi}(y),
\overline{\psi}(y)
)$ 
satisfies that \[
\|w\|\geq \max\{
\underset{(A',\overline{\eta}')\in U}{\sup}
h(A',\overline{\eta}'),d\},
\]
then
 \begin{align}\label{bddy}
\|y\|\leq g_{(A,\overline{\eta})}(\|w\|)\leq g(\|w\|)
\qquad\text{(because $\|w\|\geq
\underset{(A',\overline{\eta}')\in U}{\sup}h(A',\overline{\eta}')$)},
\end{align}
and hence
\begin{align}
\label{plybdsol}
   &\|(\varphi\circ f_1(y),...,\varphi\circ f_k(y))\|\\\notag
   \leq& \sum_{i=1}^k|\varphi\circ f_i(y)|\\\notag
   \leq& k|\varphi 
   (\exp_s(\|y\|))| \qquad&\text{by i.}\\\notag
   \leq& k|\varphi(\exp_s(g(\|w\|)))|\qquad&
   \text{ by (\ref{bddy})}\\\notag
   \leq& k
   |\varphi(\exp_{2s}(\|w\|))|\qquad&
   \text{ by ii. }\\
   <&\dfrac{1}{2}\|w\|.\qquad&
   \text{ by iii. }\notag
\end{align}

But because $\|w\|\geq d$, by iv., this is impossible.
Hence,  for all $y \in
Z(P\circ A(\overline{X},\overline{l},
\overline{\epsilon},
\delta,
\overline{\varphi},
\overline{\psi}
)-\overline{\eta})$,
  $
\|\overline{l},
\overline{\epsilon},
\delta,
\overline{\varphi}(y),
\overline{\psi}(y)
\|
\leq\max\{\underset{(A',\overline{\eta}')\in U}{\sup} h(A',\overline{\eta}'),d\} $.

Also, because $g_{A,\overline{\eta}}$ and $g$ are increasing,
\begin{align*}
    \|y\|
    &\leq g_{A,\overline{\eta}}(\|
\overline{l},
\overline{\epsilon},
\delta,
\overline{\varphi}(y),
\overline{\psi}(y)
\|)
\\
&\leq g_{A,\overline{\eta}}(\max\{
\underset{(A',\overline{\eta}')\in U}{\sup}h(A',\overline{\eta}'),d\} )
\\
&\leq g(\max\{\underset{(A',\overline{\eta}')\in U}{\sup}h(A',\overline{\eta}'),d\} ).
\end{align*}

Hence, given $(A_0,\overline{\eta}_0)\in G$, we found
an open neighborhood $U\subseteq\overline{U}\subseteq D_1$
and a real number  $r=g(\max\{
\underset{(A',\overline{\eta}')\in U}{\sup}
h(A',\overline{\eta}'),d\})$  such that for all $(\overline{l},
\overline{\epsilon},
\delta)
\in[-1,1]^{2n+2}$,
\[
\bigcup_{(A,\overline{\eta})\in U}
Z( P\circ A(\overline{X},\overline{l},
\overline{\epsilon},
\delta,
\overline{\varphi},
\overline{\psi})
-\overline{\eta})
\subseteq B_{\leq r}(\mathbf{0}).
\] 

For 2.,
given a fixed $(A,\overline{\eta})\in G$, by 1.,
$Z(P\circ A(\overline{X},\overline{l},
\overline{\epsilon},
\delta,
\overline{\varphi},
\overline{\psi})
-\overline{\eta})
\subseteq B_{\leq r}(\mathbf{0})$
for some $r\in\mathbb{R}$.
So the system is reduced to a system of lower complexity.
More precisely,
 since for all $y\in Z(P\circ A(\overline{X},\overline{l},
\overline{\epsilon},
\delta,
\overline{\varphi},
\overline{\psi})
-\overline{\eta})$,
$\|y\|\leq r$ for some $r\in\mathbb{R}$,
we may replace  each $\varphi$-monomial (resp. $\varphi'$-monomial) $\varphi\circ f_i$  (resp. $\varphi'\circ g_j$  )
by $\alpha\circ f_i$ (resp. $\gamma\circ g_j$) for some smooth $\alpha$ (resp. $\gamma$) definable in $\mathbb{R}_{an,\exp}$ with domain 
$\mathbb{R}$ that extends $\varphi|_{[-D,D]}$ (resp. $\varphi'|_{[-D,D]}$) for some well-chosen $D$.

For 3., apply 1. to each point of the compact set $\sigma([0,1])$.

\end{proof}

\begin{theorem}
    \label{rtb}
    Let  $P,R,G$ be defined as in (\ref{gensys}), (\ref{rset}), (\ref{gset}) resp..
    Fix a cell decomposition $D_1\cup\cdots\cup D_q$ of $G$ in 
  $\mathbb{R}_{an,\exp}$ satisfying 1. in  Lemma \ref{indbdd}.
  Suppose that  the set $R$ defined in (\ref{rset})
  has a triangulation $\{\Delta_\alpha\}$ such that 
  \begin{itemize}
      \item for all $(A,\overline{\eta},\overline{l},\overline{\epsilon},\delta)
  \in \overline{G}\times[-1,1]^{2n+2}$,
  there is a neighborhood of it that intersects finitely many $\Delta_\alpha$'s;
  \item for each $\alpha$, there is $i\in[q]$ such that $\Delta_\alpha\subseteq D_i\times [-1,1]^{2n+2}$.
  \end{itemize}
    
    Then
 there is $N\in\mathbb{N}$ such that\ignore{
  for all  $(\bar{l},\bar{\epsilon},\delta)\in[-1,1]^{2n+2}$ and all $(A,\overline{\eta})\in G$,
 the system \[
 P\circ A
 (\overline{X}, \overline{l},\overline{\epsilon},\delta,
\overline{\varphi},
\overline{\psi}
)
 =\overline{\eta}\]
 has $\leq N$ many non-singular solutions;}
  for all  $(\bar{l},\bar{\epsilon},\delta)\in[-1,1]^{2n+2}$,
 the system $P(\overline{X}, \overline{l},\overline{\epsilon},\delta,
\overline{\varphi},
\overline{\psi}
)=\mathbf{0}$ has $\leq N$ many non-singular zeroes.
\end{theorem}
\begin{proof}(cf. proof of Bezout's theorem in  \cite[Lemma~11.5.1]{bochnak2013real})

 Let $P$ be a system of the form (\ref{gensys}), and let $G,R$ be defined as in (\ref{gset}), (\ref{rset}) resp..
\ignore{
We first show that 1. implies 2.
Let $m=3n+k+u+2$.
Since by Lemma \ref{bddden},
$G$ is dense in $GL_m(\mathbb{R})\times[-1,1]^{n}$,
we can choose $(A_i,\overline{\eta}_i)_{i\in\mathbb{N}}$  a sequence in $G$ converging to  $(id,\mathbf{0})$.
Given  $(\bar{l},\bar{\epsilon},\delta)\in[-1,1]^{2n+2}$,
since by 1., for all $i\in\mathbb{N}$,
$ P\circ A_i
 (\overline{X}, \overline{l},\overline{\epsilon},\delta,
\overline{\varphi},
\overline{\psi}
)=\overline{\eta}_i$ has $\leq N$ many non-singular zeroes,
by Implicit Function Theorem, 
$ P
 (\overline{X}, \overline{l},\overline{\epsilon},\delta,
\overline{\varphi},
\overline{\psi}
)=\mathbf{0}$ has $\leq N$ many non-singular zeroes.

Now we prove 1.}

By assumption,
we get a triangulation of $R$, say $\mathscr{R}$ such that for each $(id,\mathbf{0},\overline{l},\overline{\epsilon},\delta)
\in\{(id,\mathbf{0})\} \times[-1,1]^{2n+2}$, 
there is a neighborhood $U_{(id,\mathbf{0},\overline{l},\overline{\epsilon},\delta)}$
of  $(id,\mathbf{0},\overline{l},\overline{\epsilon},\delta)$
such that only finitely many members in $\mathscr{R}$ intersect $U_{(id,\mathbf{0},\overline{l},\overline{\epsilon},\delta)}$.
Since $\{(id,\mathbf{0}
)\}\times[-1,1]^{2n+2}$ is compact, there exist $(\overline{l}_1,\overline{\epsilon}_1,\delta_1),...,(\overline{l}_s,\overline{\epsilon}_s,\delta_s)\in[-1,1]^{2n+2} $ 
such that \[
\{(id,\mathbf{0})\}\times[-1,1]^{2n+2}\subseteq U_{(id,\mathbf{0},\overline{l}_1,\overline{\epsilon}_1,\delta_1)}\cup\cdots\cup
U_{(id,\mathbf{0},\overline{l}_s,\overline{\epsilon}_s,\delta_s)}.
\]
Let $\mathscr{R}'$ be the collection of 
members of $\mathscr{R}$ that intersect $U_{(id,\mathbf{0},\overline{l}_1, \overline{\epsilon}_1,\delta_1)}\cup\cdots\cup
U_{(id,\mathbf{0},\overline{l}_s, \overline{\epsilon}_s,\delta_s)} $.
By local finiteness of $\mathscr{R}$ and the choice of the neighborhoods,
$\mathscr{R}'$ is finite.
\begin{claim}
 \label{appcl}
     For all $(id,\mathbf{0},\overline{l},\overline{\epsilon},\delta)
\in \{(id,\mathbf{0})\}\times
[-1,1]^{2n+2}
$,
there is $F\in\mathscr{R}'$
such that $(id,\mathbf{0},\overline{l},\overline{\epsilon},\delta)
\in\overline{F}$.
 \end{claim}
 \begin{proof}
     [Proof of Claim \ref{appcl}]
This is because $R$ is dense in $GL_m(\mathbb{R})\times[-1,1]^{3n+2}$ by Lemma \ref{subandenreg}.
  \end{proof}
\begin{claim}\label{pbrt}
    Given $(\overline{l},\overline{\epsilon},\delta)
\in
[-1,1]^{2n+2}
$
and $F\in \mathscr{R}'$ 
such that $(id,\mathbf{0},\overline{l},\overline{\epsilon},\delta)\in \overline{F}$,
if $(A_F,\overline{\eta}_F,\overline{l}_F,\overline{\epsilon}_F,\delta_F)\in F$ and $P\circ A_F(\overline{X}, \overline{l}_F,\overline{\epsilon}_F,\delta_F,
\overline{\epsilon},
\overline{\varphi},
\overline{\psi}
)-\overline{\eta}_F$ has $\leq N_F$ non-singular zeroes, then $P(\overline{X}, \overline{l},\overline{\epsilon},\delta,
\overline{\varphi},
\overline{\psi}
)$  has $\leq  N_F$  non-singular zeroes.
\end{claim}
\begin{proof}
     [Proof of Claim \ref{pbrt}]
 Let  $(id,\mathbf{0},\overline{l},\overline{\epsilon},\delta)
\in
\{(id,\mathbf{0})\}\times[-1,1]^{2n+2}
$ and $F\in \mathscr{R}'$ 
such that $(id,\mathbf{0},\overline{l},\overline{\epsilon},\delta)\in \overline{F}$.
Let $(A_F,\overline{\eta}_F,\overline{l}_F,\overline{\epsilon}_F,\delta_F)\in F$. 
Because $F$ is a simplex, we can find $\sigma:[0,1]\to \overline{F}$  a  $C^1$ path such that
\begin{itemize}
    \item 
$\sigma(0)=(id,\mathbf{0},\overline{l},\overline{\epsilon},\delta)$;
\item
$\sigma(1)=(A_F,\overline{\eta}_F,\overline{l}_F,\overline{\epsilon}_F,\delta_F)$;
\item 
for all $t\in (0,1]$, $\sigma(t)\in F$.
\end{itemize}
For notational convenience, we write
$\sigma(t)=(A(t),\overline{\eta}(t),\overline{l}(t),\overline{\epsilon}(t),\delta(t))$.

Suppose that $P(\overline{X},\overline{l}, \overline{\epsilon},\delta,
\overline{\varphi},
\overline{\psi}
)
=P\circ A(0)(\overline{X},
\overline{l}(0),\overline{\epsilon}(0),
\delta(0),
\overline{\varphi},
\overline{\psi}
)
-\overline{\eta}(0)$ has $\geq  N_F+1$  non-singular zeroes.
We will show that
\begin{align*}
S=\{t\in[0,1]:
P\circ A(t)(\overline{X}, 
\overline{l}(t),\overline{\epsilon}(t),
\delta(t),
\overline{\varphi},
\overline{\psi}
)-\overline{\eta}(t)
\\\text{
 has $\geq   N_F+1$  non-singular zeroes}\}=[0,1],
\end{align*}
 thus contradicting the assumption.

Let $B=\{b\in [0,1]:[0,b]\subseteq S\}$.
Then 
$0\in B$ because $0\in S$ by assumption.
$B$ is open by the Implicit Function Theorem.
We now prove that $B$ is closed.

Suppose $b\in \overline{B}\setminus B$.
Since $0\in B$ and $B$ is open, we have that
 $b>0$,
 and there is an increasing sequence of positive numbers $(b_n)_{n\in\mathbb{N}}$ in $B$ converging to $b$.
For each $n\in\mathbb{N}$, fix $N_F+1$ distinct non-singular zeroes for
$P\circ A(b_n)(\overline{X},
\overline{l}(b_n),\overline{\epsilon}(b_n),
\delta(b_n),
\overline{\varphi},
\overline{\psi}
)
-\overline{\eta}(b_n)$, say $x_{n,1},...,x_{n,N_{F}+1}$.
By Lemma \ref{indbdd},
there is a compact ball such that 
\[\underset{n\in\mathbb{N}}{\bigcup}Z(P\circ A(b_n)(\overline{X},
\overline{l}(b_n),\overline{\epsilon}(b_n),
\delta(b_n),
\overline{\varphi},
\overline{\psi}
)
-\overline{\eta}(b_n)
)\] is contained in that ball.
So taking a subsequence if necessary, we may assume that the sequences $(x_{n,1})_{n\in\mathbb{N}}$,
..., $(x_{n,N_{F}+1})_{n\in\mathbb{N}}$ converge to points $x_1,...,x_{N_F+1}$.
By continuity of $\sigma$, these are zeros of $P\circ A(b)(\overline{X},
\overline{l}(b),\overline{\epsilon}(b),
\delta(b),
\overline{\varphi},
\overline{\psi}
)
-\overline{\eta}(b)$.
By assumption, 
$(A(b),\overline{\eta}(b),\overline{l}(b),\overline{\epsilon}(b),\delta(b))\in F$,
so
all zeroes of \\
$P\circ A(b)(\overline{X},
\overline{l}(b),\overline{\epsilon}(b),
\delta(b),
\overline{\varphi},
\overline{\psi}
)
-\overline{\eta}(b)$  are non-singular. 
Also, due to Implicit Function Theorem,
it has $\geq N_{F}+1$ non-singular zeros.
So $b\in B$, a contradiction. 
Hence $\overline{B}\setminus B=\emptyset$ and $B$ is closed.
By connectedness of $[0,1]$, $B=[0,1]$.
It follows that
$P\circ A(\overline{X},\overline{l}, \overline{\epsilon},\delta,
\overline{\varphi},
\overline{\psi}
)
=P\circ A(0)(\overline{X},
\overline{l}(0),\overline{\epsilon}(0),
\delta(0),
\overline{\varphi},
\overline{\psi}
)
-\overline{\eta}(0)$ has $\leq   N_{F}$  many non-singular zeroes.
\end{proof}
For each $F\in \mathscr{R}'$, fix a point $(A_F,\overline{\eta}_F,\overline{l}_F,\overline{\epsilon}_F,\delta_F)\in F$.
By 2. of Lemma \ref{indbdd} and inductive hypothesis,
there is $N_F\in \mathbb{N}$ such that
$P\circ A_F(\overline{X}, \overline{l}_F,\overline{\epsilon}_F,\delta_F,
\overline{\epsilon},
\overline{\varphi},
\overline{\psi}
)-\overline{\eta}_F$ has $\leq N_F$  non-singular zeroes.
Now take $N=
\max\{N_F:F\in\mathscr{R}'\}$.
By Claim \ref{appcl}  and Claim \ref{pbrt},  for all $(\overline{l},\overline{\epsilon},\delta)
\in[-1,1]^{2n+2}$,
$P(\overline{X}, \overline{l},\overline{\epsilon},\delta,
\overline{\varphi},
\overline{\psi}
)$ has $\leq  N$  non-singular zeroes.
\end{proof}
\begin{remark}
By \cite[Chapter~8, 2.9]{van1998tame},
    the assumption in Theorem \ref{rtb} is a necessary condition for the o-minimality of $\mathbb{R}_{an,\exp,\varphi}$ because $R$ is definable in $\mathbb{R}_{an,\exp,\varphi}$.
  
\end{remark}
\subsection{Remarks on the Regularity Assumption}

 \begin{definition}
    Let $C$ be a cell definable in $\mathbb{R}_{an,\exp}$.
    A (not necessarily definable) set $X\subseteq C$ is 
    \textit{subanalytic in $C$} if for each $x\in C$, there is a neighborhood $U$ of $x$ in $C$ such that  $X\cap U$ is the projection of a bounded semi-analytic set.
\end{definition}
\begin{remark}
    Fix a cell decomposition $D_1\cup\cdots\cup D_q$ of $G$ satisfying Lemma \ref{indbdd} and
a cell decomposition $C_1\cup\cdots\cup C_v$ of 
$G\times
[-1,1]^{2n+2}$ such that for all $i\in [v]$ there is $j\in [q]$ satisfying that $C_i\subseteq D_j\times
[-1,1]^{2n+2}$.
For all $i\in [v]$,
       $R_i:=R\cap C_i$ is subanalytic in $C_i$:

       Observe that the set $R$ can be written as 
\begin{align*}
    R=\{(A,\overline{\eta},\overline{l},\overline{\epsilon},\delta)
    \in
    G\times[-1,1]^{2n+2}: \forall \overline{x}\,
    {\text{with $\|x\|\leq g(A,\overline{\eta})$, }
    }\\
    P\circ A(\overline{x},\overline{l},
\overline{\epsilon},
\delta,
\overline{\varphi},
\overline{\psi})
\neq
\overline{\eta}
\text{ or}
\det J_{\overline{x}}\,(
P\circ A(\overline{X},\overline{l},
\overline{\epsilon},
\delta,
\overline{\varphi},
\overline{\psi}))
\neq 0\}
\end{align*}
Locally at each point, the complement $R_i^c$ of $R\cap C_i$ in 
$C_i$ is the projection of a \textit{bounded} semi-analytic set,
because 
by Lemma \ref{indbdd},
for each $(A,\overline{\eta})\in D_j$,
there exist a neighborhood $U$ of $(A,\overline{\eta})$ in $D_j$ and a compact ball that contains $\bigcup_{(A,\overline{\eta})\in U}
Z( P\circ A(\overline{X},\overline{l},
\overline{\epsilon},
\delta,
\overline{\varphi},
\overline{\psi}
)
-\overline{\eta})$.
Hence for each $p\in R_i^c$,
we can find a \textit{bounded} semi-analytic set such that locally around $p$, $R^c$ is the projection of it.

To get a locally finite triangulation needed in Theorem \ref{rtb},
    one possible way is to show that  $R$ is subanalytic in the sense of \cite[DEFINITION 3.3.]{hironaka11974triangulations}.
    Then \cite[Subanalytic triangulation]{hironaka11974triangulations} gives a triangulation we need.
    What we have here is local subanalycity of $R$ in $G\times[-1,1]^{2n+2}$.
    But to get subanalycity and triangulation in \cite{hironaka11974triangulations}, 
    we need subanalycity of $R$ in $\overline{G}\times [-1,1]^{2n+2}$.

\end{remark}

\newpage
\section{O-minimality of 
$\mathbb{R}_{an,\exp,\varphi}$
} \label{proof}
\indent

This section is rephrasing Wilkie's work \cite{wilkie1999theorem}
 and  Milnor's work \cite{milnor1963morse} for expository purpose.

\subsection{Wilkie's Test of o-Minimality}
\indent

In \cite{wilkie1999theorem}, Wilkie gave a test for o-minimality. 
Combining this test with Khovanskii's theory \cite{khovanskiui1991fewnomials}, Wilkie  proved o-minimality of $\mathbb{R}_{\exp}$ in \cite{wilkie1999theorem}.
 To deduce the o-minimality of $\mathbb{R}_{an,\exp,\varphi}$ from the existence of a uniform bound for the number of solutions to systems $P$ in (\ref{gensys}),
we will use Wilkie's test instead of the original definition of o-minimality.
\begin{definition}\label{gammaA}
    \cite[Definition~1.2]{wilkie1999theorem}
    Suppose $n \geq 1$ and $A \subseteq \mathbb{R}^n$. Then $\gamma(A)$ denotes the smallest natural
number $N$ with the following property: 
for any aﬃne subspace $X$ of $\mathbb{R}^n$ we
have $A\cap X= A_1 \cup\cdots\cup A_N$ for some connected subsets  (in the Euclidean topology) $A_1,\ldots,A_N$ of $\mathbb{R}^n$.
If no such $N$ exists we write $\gamma(A) = \infty$.
\end{definition}
\begin{fact}\label{omintest}
    \cite[Theorem~1.9]{wilkie1999theorem}
    Let $\mathcal{M}$ be any expansion of the ordered ring of real numbers by
$C^{\infty}$ functions. 
Suppose that for each $n \geq 1$ every quantifier-free, $\mathcal{M}$-definable
(with parameters) subset $A$ of $\mathbb{R}^n$ satisfies $\gamma(A) < \infty$.
Then $\mathcal{M}$ is o-minimal (in the usual model-theoretic sense that every $\mathcal{M}$-definable (with parameters) subset of
$\mathbb{R}$ has finite boundary).
\end{fact}

Before we deal with the geometry, we first simplify the model-theoretic part.
In Fact \ref{omintest}, we need to look at all quantifier-free sets. 
But as in \cite[Theorem~1.9]{wilkie1999theorem},
we only need to bound the number of connected components for term-definable hypersurfaces.
\subsection{Reducing the Quantifier-Free Formulas}
 \indent

 An \textit{affine subspace} of $\mathbb{R}^n$ is $\mathbf{v}+Z(L)$ where $\mathbf{v}\in\mathbb{R}^n$ and $L:\mathbb{R}^n\to \mathbb{R}^n$ is a linear transformation.
 So each affine subspace of $\mathbb{R}^n$ is the zero set of 
 \[
 \begin{cases}
     l_{1,1}X_1+\cdots+l_{n,1}X_n&=l_1
     \\.\\.\\.\\
     l_{1,n}X_1+\cdots+l_{n,n}X_n&=l_k
 \end{cases}
 \]
 for some $(l_{1,1},\ldots,l_{n,1},l_1)$,$\ldots$,
$(l_{1,n},\ldots,l_{n,n},l_n)\in[-1,1]^{n+1}$

 Let $\mathcal{R}$ be $\mathbb{R}_{an,\exp,\varphi}$.
 In the model $\mathcal{R}$,
 a quantifier-free formula is of the form 
 \[
 \underset{i=1}{\overset{s}{\bigcup}}
\underset{j=1}{\overset{r_i}{\bigcap}}\,
\{\overline{x}\in\mathbb{R}^n:
[F_{i,j}(\overline{x})\, *_{i,j}\, 0]^{\epsilon_{i,j}}\}
 \]
 where $F_{i,j}$ are term-definable functions in $\mathcal{R}$,
$\epsilon_{i,j}\in\{0,1\}$,
and $*_{i,j}\in\{ =,>\}$, for $i = 1,\ldots,s$ and $j = 1,\ldots,r_i$.
 (This is known as the disjunctive normal form in model theory.)

 Observe that $\neg (F>0)$ iff $F\leq 0$,
and $F<0$ iff $-F>0$.
Also $\neg (F=0)$ iff $(F>0)\,\vee\, (-F>0)$.
Hence, a quantifier-free formulas is of the form 
\[
\underset{i=1}{\overset{s}{\bigcup}}
\underset{j=1}{\overset{r_i}{\bigcap}}\,
\{\overline{x}\in\mathbb{R}^n:
F_{i,j}(\overline{x})\, *_{i,j}\, 0\}
\]
 where $F_{i,j}$ are term-definable functions in $\mathcal{R}$,
and $*_{i,j}\in\{ =,>\}$, for all $i \in\{ 1,\ldots,s\}$ and $j \in\{ 1,\ldots,r_i\}$.

The proof of \cite[Theorem~1.9]{wilkie1999theorem} gave a way to write a quantifier-free formula as
$\pi(Z(F))$ where $F$ is some term-definable function:
Given a quantifier-free formula of the form
\begin{equation}
\underset{i=1}{\overset{s}{\bigcup}}
\underset{j=1}{\overset{r_i}{\bigcap}}\,
\{\overline{x}\in\mathbb{R}^n:
F_{i,j}(\overline{x})\, *_{i,j}\, 0\}
\label{qfform}
\end{equation}
 where $F_{i,j}$ are term-definable functions in $\mathcal{R}$
and $*_{i,j}\in\{ =,>\}$ for all $i \in\{ 1,\ldots,s\}$ and $j \in\{ 1,\ldots,r_i\}$,
it is $\pi(Z(F))$.
Here $F$ is defined by
\[
\underset{i=1}{\overset{s}{\prod}}
\,\underset{j=1}{\overset{r_i}{\sum}}
\,\overline{F}_{i,j}^2,
\]
where\begin{equation}\label{fdef}
    \begin{cases}
    \overline{F}_{i,j}=F_{i,j}\qquad
    &\text{ if $*_{i,j}$ is $=$}\\
\overline{F}_{i,j}=F_{i,j}\cdot X_{i,j}^2-1
\qquad&\text{ if $*_{i,j}$ is $>$.} 
\end{cases}
\end{equation}
Here, the $X_{i,j}$'s are new variables.
\begin{definition}\label{qfaff}
Given $A$ a quantifier-free definable set in $\mathcal{R}$ of the form (\ref{qfform}), the \textit{affine family associated to $A$} is the family of functions
\[
  \mathcal{F}_A:=\{F_L:L\text{ is an affine plane}\}\]
  where $F_L=\underset{i=1}{\overset{s}{\prod}}
\,\underset{j=1}{\overset{r_i}{\sum}}
\,\overline{F}_{i,j}^2+
\,\underset{m=1}{\overset{k}{\sum}}
(l_{1,m}x_1+\cdots+l_{n,m}x_n-l_m)^2)$, and the $(n+1)$-tuples $
(l_{1,1},\ldots,l_{n,1},l_1)$,$\ldots$,
$(l_{1,n},\ldots,l_{n,k},l_n)\in[-1,1]^{n+1}$ define the affine plane $L$.
\end{definition}

Observe that if $Z(F)$ has $\leq N$ many connected components, then $\pi(Z(F))$ has $\leq N$ many connected components, because $\pi$ is continuous.
It follows that the problem of proving o-minimality of $\mathbb{R}_{an,\exp,\varphi}$ is reduced to proving the following:
\begin{problem}\label{rdpm}
Given a quantifier-free definable set $A\subseteq\mathbb{R}^n$ of the form (\ref{qfform}), 
there is $N(A)\in\mathbb{N}$ such that for all $n$-tuples $
(l_{1,1},\ldots,l_{n,1},l_1)$,$\ldots$,
$(l_{1,n},\ldots,l_{n,k},l_n)\in[-1,1]^{n+1}$, the set
\[
Z(\underset{i=1}{\overset{s}{\prod}}
\,\underset{j=1}{\overset{r_i}{\sum}}
\,\overline{F}_{i,j}^2+
\,\underset{m=1}{\overset{k}{\sum}}
(l_{1,m}x_1+\cdots+l_{n,m}x_n-l_m)^2)
\] has $\leq N(A)$ many connected components.
The $\overline{F}_{i,j}$'s are defined as in (\ref{fdef}).
\end{problem}

\subsection{Weak Morse Inequalities}
\indent

To use Fact \ref{omintest}, we need to count connected components of subsets in $\mathbb{R}^n$. 
Imitating \cite{khovanskiui1991fewnomials} and \cite{milnor1964betti},
we will use Morse theory to count connected components.
We briefly explain Morse theory in this subsection and list the facts we will use.

For a topological space $X$, the zero-th betti number is the same as the number of path-connected components. 
(See e.g. \cite[Proposition~2.7]{hatcher}.)
Hence, to bound the number of connected components, it suffices to bound the sum of betti numbers. (By \cite[Corollary~9.3]{gamelin2013introduction}, each connected component
of $X$ is a union of path-connected components of $X$.)

Morse theory provides a good way to bound the number of connected components of a compact manifold.
See \cite{milnor1963morse}  for more details on Morse theory.

\begin{definition}\cite[Chapter~I, Section~2]{milnor1963morse}
Let $M$ be a  manifold.
Let  $f:M\to \mathbb{R}$ be a smooth function.
A \textit{critical point} of $f$ is $p\in M$ such that $df_p : T_p M \to T_{f(p)}\mathbb{R}$ is zero.
  
  A critical point $p\in M$ is called \textit{non-degenerate} if and only if the matrix
\[
\left( \frac{\partial^2 f}{\partial x^i \partial x^j}(p) \right)
\]
is non-singular.
\end{definition}
\begin{definition}\cite[p.29]{milnor1963morse}
    Let $M$ be a smooth manifold and $f$ a differentiable function
on $M$. $f$ is called a \textit{Morse function} if all of its  critical points are non-degenerate.
\end{definition}
\begin{fact}
\label{nondegiso}
\cite[Corollary~2.3]{milnor1963morse}
    Non-degenerate critical points are isolated.
\end{fact}
\begin{fact}
\label{morine}
\cite[Theorem~5.2]{milnor1963morse}
   (Weak Morse Inequalities).
  Let $M$ be a compact manifold.
  Let $f:M\to \mathbb{R}$ be a Morse function on $M$.
  Then 
  \[
  \text{the sum of Betti number of $M$} \leq \text{the number of critical points of $f$}.
  \]
\end{fact}

Hence, by Fact \ref{omintest} and Fact \ref{morine}, the problem of proving o-minimality is related to the problem of bounding zeroes of non-degenerate systems of equations.

\subsection{Reducing to Compact Smooth Manifolds}
\indent

The Weak Morse Inequality (Fact \ref{morine}) gives a convenient way to bound the number of connected components of a compact smooth manifold,
but given an arbitrary term-definable $F$, $Z(F)$ is neither necessarily compact nor necessarily a smooth manifold. 
Milnor's method in \cite{milnor1964betti}  solved both of these issues.
\ignore{
As in \cite{milnor1963morse}, unless specified, the homology and cohomology groups have coefficients in some fixed field, say $\mathbb{Q}$.
}
\ignore{
 
 \subsubsection{Triangulability of Term-Definable Hypersufaces}
 \indent
 
 Because of Fact \ref{subanstra} and analyticity of $\varphi$, we have the following:
 
 \begin{fact}\label{trian}
     Given a term-definable $F$ in $\mathbb{R}_{an,\exp,\varphi}$,
     $Z(F)$ is triangulable.
 \end{fact}
 }
 \subsubsection{Reducing to Compact Sets}\label{redcomp}
 \indent

As in \cite[Theorem~2]{milnor1964betti}, Problem \ref{rdpm} can be reduced to compact sets because of the following fact:
 \begin{fact}
 \label{drlim}\cite[Proposition~3.33]{hatcher} 
     If a space $X$ is the union of a directed set of subspaces $X_\alpha$ with
the property that each compact set in $X$ is contained in some $X_\alpha$, then the natural
map \[
\lim_{\longrightarrow}
H_i(X_\alpha; G)\to H_i(X; G)
\]
is an isomorphism for all $i$ and $G$.
 \end{fact}
\ignore{
 \begin{fact}\label{limunionr}
 Let $F$ be a term-definable function. 
     For each $r\in\mathbb{N}$,
     let $K_r=Z(F)\cap B_{<r}(0)$.
     Given $i\in\mathbb{N}$,
if  $M\in\mathbb{N}$ satisfies that for all $r\in\mathbb{N}$,
rank $ H_i(K_r)\leq M$.
     
     Then \[
     \text{rank} \, H_i(Z(F))\leq M
     \]
 \end{fact}
 \begin{proof}
     By Fact \ref{drlim},
rank $H_i(Z(F))$ $=$ rank $\underset{\longrightarrow}{\lim} H_i(K_r)$ $\leq$ $M$ .

 \end{proof}
 }
 \subsubsection{Reducing to Compact Smooth Manifolds}\label{redcsm}
 \indent
 
 Given a term-definable $F$,
 $Z(F)$ is not necessarily a smooth manifold unless $0$ is a regular value of $F$.
 We follow the proof of \cite[Theorem~2]{milnor1964betti} to reduce 
 Problem \ref{rdpm} to the case of compact smooth manifolds.
The argument in \cite{milnor1964betti} works here because all spaces we will mention are triangulable by \cite[Subanalytic triangulation]{hironaka11974triangulations}.

 C$\check{\text{e}}$ch cohomology is used because C$\check{\text{e}}$ch cohomology has the continuity property.
To use C$\check{\text{e}}$ch cohomology, we
relate the C$\check{\text{e}}$ch cohomology groups with the singular homology groups.
 \begin{fact}
 \label{cechsing}\cite[Proposition~6.12]{dold2012lectures}

    If \( Y \subseteq  X \) is a pair of ENR’s then \( \check{H}^*(X, Y) \approx H^*(X, Y) \),
i.e., for ENR’s  C$\check{\text{e}}$ch cohomology coincides with singular cohomology.
 \end{fact}
 \begin{fact}\cite[Proposition~8.12]{dold2012lectures}\label{rnenr}
 
     If $X\subseteq \mathbb{R}^n$ is locally compact and locally contractible then $X$ is an ENR.

     It follows from 
     \cite[Subanalytic triangulation]{hironaka11974triangulations},
     Fact \ref{cechsing}, \ref{rnenr} that for all spaces we will mention, we can use  C$\check{\text{e}}$ch cohomology and singular cohomology interchangably.
 \end{fact}
 
 The Universal Coefficients Theorem relates the rank of homology and cohomology groups.
 \begin{fact}\label{dual}
 \cite[Theorem~3.2.]{hatcher} 
 (Universal Coefficients Theorem)
 
 If a chain complex \( C \) of free abelian groups has homology groups  
\( H_n(C) \), then the cohomology groups \( H^n(C; G) \) of the cochain complex \( \operatorname{Hom}(C_n; G) \)  
are determined by split exact sequences  

\[0 \to \operatorname{Ext}(H_{n-1}(C); G) \to H^n(C; G) \xrightarrow{h} \operatorname{Hom}(H_n(C); G) \to 0\]

     As a result, for a space $X$, if each $H_i(X;\mathbb{Z})$ is finitely generated,
     say $X$ is a finite simplicial complex or a compact triangulable space (\cite[Theorem~2.27]{hatcher}),
     then for each $i$, $rank \,\,H^i(X;\mathbb{Z})=rank\,\, H_i(X;\mathbb{Z})$.
 \end{fact}
 \ignore{
The following fact is immediate from Fact \ref{cechctn}, \ref{cechsing}, \ref{dual}.
 \begin{fact}\label{intlimhom}
Let $V$ be a  CW-complex in $\mathbb{R}^n$.
Suppose $V=\underset{r\in\mathbb{N}}{\bigcap}K_r$ where $K_r$ is a decreasing sequence of compact sets in $\mathbb{R}^n$.
       Given $i\in\mathbb{N}$,
if  $M\in\mathbb{N}$ satisfies that for all $r\in\mathbb{N}$,
rank $ H_i(K_r)\leq M$.
     
     Then \[
     \text{rank} \, H_i(V;\mathbb{Q})\leq M
     \] 
 \end{fact}
  \begin{fact}\label{limunionr}
 Let $F$ be a term-definable function. 
    Suppose $V=\underset{r\in\mathbb{N}}{\bigcap}K_r$ where $K_r$ is a decreasing sequence of compact sets in $\mathbb{R}^n$.
     Given $i\in\mathbb{N}$,
if  $M\in\mathbb{N}$ satisfies that for all $r\in\mathbb{N}$,
rank $ H_i(K_r;\mathbb{Q})\leq M$.
     
     Then \[
     \text{rank} \, H_i(Z(F);\mathbb{Q})\leq M
     \]
 \end{fact}
 \begin{proof}
     Immediate from Fact \ref{trian} and Fact \ref{intlimhom}.
 \end{proof}}
 \subsection{Bounding \# of Connected Components of Term-Definable Hypersurfaces}

 \subsubsection{Applying Weak Morse Inequalities}
\indent

 \begin{fact}\cite[Corollary~3.45]{hardt1976triangulation}
     (Alexander Duality.)
     
     If $K$ is a compact, locally contractible, nonempty, proper subspace
of $S^n$, then \[\Tilde{H}_{i}(S^n
\setminus K; \mathbb{Z})
\approx \Tilde{H}^{n-i-1}(K;\mathbb{Z})\]
for all $i$.

Here, $\Tilde{H}_*$ 
(reps. $\Tilde{H}^*$) denotes  the reduced homology (resp. reduced cohomology).
For a space $X$,
\[{H}_i(X;G)\approx
\begin{cases}
\Tilde{H}_i(X;G)
      &\qquad\text{if $i>0$}\\
         \Tilde{H}_0(X;G)\oplus\mathbb{Z} &\qquad\text{if $i=0$}.
\end{cases}
\]
(See \cite[Chapter~2, Section~2.1, p.110]{hatcher}.)
\[\Tilde{H}^i(X;G)\approx
\begin{cases}
    {H}^i(X;G)
    &\qquad\text{if $i>0$}\\
  \operatorname{Hom}( \Tilde{H}_0(X), G)
     &\qquad\text{if $i=0$}.
\end{cases}
\]
(See \cite[Chapter~3, Section~3.1, p.199]{hatcher}.)
\ignore{
Also, because of \cite[3A.4, 3A.5, 3A.6]{hatcher},
when $X$ is a compact triangulable space, 
we do not distinguish $rank\,\, H_i(X;\mathbb{Z})$ and $\dim_{\mathbb{Q}}\,\, H_i(X;\mathbb{Q})$.}
 \end{fact}
 \ignore{
 \begin{fact}
Given a continuous function $F:\mathbb{R}^n\to \mathbb{R}$,     $\partial\{x\in\mathbb{R}^n:F(x)\leq\delta\}=\{x\in\mathbb{R}^n:F(x)=\delta\}$.

 \end{fact}
 }
 \begin{theorem}\label{bddcsmconn}
    Let $A\subseteq\mathbb{R}^n$ be a quantifier-free definable set.
\ignore{Fix $k\in\mathbb{N}$.}
     Let $L$ be an affine subspace in $\mathbb{R}^n$\ignore{defined by $k$ many $(n+1)$-tuples}. 
     Suppose that the regularity assumption in 
     Theorem \ref{rtb} holds.
     
    Then there is $N\ignore{(k)}\in\mathbb{N}$ such that for all $\epsilon,\delta\in(0,1)$, 
    if $\delta^2$ is a regular value for $F_L(X)^2+\epsilon (X_1^2+\cdots+X_n^2)$,
    where $F_L$ is a function of the form in Definition \ref{qfaff},
    then 
     \[\sum
     \text{rank}\, H^i(\{x\in\mathbb{R}^n:
F_L(x)^2+\epsilon(x_1^2+\cdots+x_n^2)
\leq\delta^2
     \})\leq N.
     \]
 \end{theorem}
 \begin{proof}
$\{x\in\mathbb{R}^n:
F_L(x)^2+\epsilon(x_1^2+\cdots+x_n^2)\leq\delta^2
     \}$ is compact, but not necessarily a manifold.
     $\{x\in\mathbb{R}^n:
F_L(x)^2+\epsilon(x_1^2+\cdots+x_n^2)=\delta^2
     \}$ is a manifold when $\delta^2$ is a regular value.
We can relate the sum of Betti numbers of these two spaces by repeating the argument in \cite[(9)~Theorem]{milnor1956immersion}.
Let $K^{\square}(\epsilon,\delta)$ denote   $\{x\in\mathbb{R}^n:
F_L(x)^2+\epsilon(x_1^2+\cdots+x_n^2)\square \delta^2
     \}$, where $\square\in\{\leq ,=,<,>\}$.
     \footnote{
     $K^{\leq }(\epsilon,\delta)$ is subanalytic because $\varphi$ is analytic. 
     }
\begin{claim}
    If $\delta^2$ is a regular value of 
    $F_L(X)^2+\epsilon (X_1^2+\cdots+X_n^2)$,
    then \[
    \sum rank\,\, H^i(K^{\leq }(\epsilon,\delta))=
    \dfrac{1}{2}  \sum rank\,\, H^i( K^{=}(\epsilon,\delta))
    \]
\end{claim}
 \begin{proof}
 Let $S^{n}$ be the one-point compactification of $\mathbb{R}^n$.
By Alexander Duality applied to $K^{\leq }(\epsilon,\delta)$, 
 \[
 \sum  rank\,\, H^i(K^{=}(\epsilon,\delta))= 
 \sum  rank \,\, H_i(S^n\setminus K^{=}(\epsilon,\delta)).
 \]
 By Mayer-Vietoris (\cite[Chapter~2, Section~2.2, p.149]{hatcher}), since $K^{<}(\epsilon,\delta)$ and $S^n\setminus K^{\leq }(\epsilon,\delta)$ are open and disjoint,
 \[
 \sum rank\,\, H_i(S^n\setminus K^{=}(\epsilon,\delta))=
 \sum rank \,\,H_i( K^{<}(\epsilon,\delta))
 +\sum rank\,\, H_i(S^n\setminus K^{\leq }(\epsilon,\delta)).
 \]
Since $ K^{\leq }(\epsilon,\delta)$ is a compact manifold with boundary, by \cite[Proposition~3.42]{hatcher} and \cite[Chapter~6, Theorem~20]{spanier2012algebraic},
\[
\sum  rank \,\,H_i( K^{<}(\epsilon,\delta))= 
\sum rank \,\,H_i( K^{\leq }(\epsilon,\delta)).
\]
By Alexander Duality applied to $ K^{\leq }(\epsilon,\delta)$,\[
\sum rank \,\,H^i( K^{\leq }(\epsilon,\delta))=
\sum rank\,\, H_i( S^n\setminus K^{\leq }(\epsilon,\delta)).
\]
It follows that 
\[
    \sum rank\,\, H^i(K^{\leq }(\epsilon,\delta))= 
    \dfrac{1}{2}  \sum rank\,\, H^i( K^{=}(\epsilon,\delta)).
    \]
 
     (See also \cite[Theorem~11.5.3]{bochnak2013real} for a different exposition.)
 \end{proof}

 Hence 
it suffices to show that there is $N\in\mathbb{N}$ such that for all  
$\epsilon,\delta\in(0,1)$, 
    if $\delta^2$ is a regular value for $F_L(X)^2+\epsilon (X_1^2+\cdots+X_n^2)$, then 
     \[
  \sum   {rank}\,\, H^i(\{x:F_L(x)^2+\epsilon(x_1^2+\cdots+x_n^2)=\delta^2
     \})\leq 
     N.
     \]
     Notice that as in \cite{milnor1964betti}, since $\delta^2$ is a regular value for  $F_L(X)^2+\epsilon (X_1^2+\cdots+X_n^2)$,
     $\{x\in\mathbb{R}^n:F_L(x)^2+\epsilon(x_1^2+\cdots+x_n^2)=\delta^2\}$ 
     is  a compact smooth manifold where Weak Morse Inequality applies.

As in \cite[Theorem~1]{milnor1964betti},
taking a rotation of the coordinates if necessary,
we may assume that the projection onto the 
$x_n$-coordinate is a Morse function on $\{x\in\mathbb{R}^n: F_L(x)^2+\epsilon (x_1^2+\cdots+x_n^2)=\delta^2\}$.

Taking first partial derivatives,
we get the system
      \begin{equation}
      \label{1der}
          \begin{cases}
        2F_L\cdot\dfrac{\partial  F_L}{\partial{X_1}}
          +2\epsilon X_1&=0\\
              ...\\...\\...\\
            2F_L\cdot    \dfrac{\partial  F_L}{\partial{X_{n-1}}}  +2\epsilon X_{n-1}&=0\\
F_L(X)^2+\epsilon(X_1^2+\cdots+X_n^2)
            &=\delta^2
          \end{cases}
      \end{equation}
      
     By Theorem \ref{rtb}, there is $N\in\mathbb{N}$ such that 
     the system (\ref{1der}) has $\leq N$ many solutions.
     Hence by the Weak Morse Inequality (Fact \ref{morine}), 
     \[
\sum  rank\,
     H^i
     (\{x\in\mathbb{R}^n:
     F_L(x)^2+\epsilon (x_1^2+\cdots+x_n^2)=\delta^2\})\leq N.
     \]
 \end{proof}
 \begin{theorem}\label{bddcomp}
Let $A\subseteq\mathbb{R}^n$ be a quantifier-free definable set.
Suppose the regularity assumption in Theorem \ref{rtb} holds.
    
Then there is $N\in\mathbb{N}$ such that  for all $r\in\mathbb{N}$,  and all  affine subspace $L$ in $\mathbb{R}^n$,
\[\sum
{rank} \,\, H^i(A\cap L\cap B_{\leq r}(0))\leq N.
\]
 \end{theorem}
 \begin{proof}
     Fix $A,L,r$ as in the statement.
     Let $F_L$ be a function of the form in Definition \ref{qfaff}.
     Construct as in \cite{milnor1964betti} a sequence 
$(\epsilon_i,\delta_i)_{i\in\mathbb{N}}$ of pairs of positive numbers in $(0,1)^2$ such that
     \begin{itemize}
         \item $\{x:F_L(X)^2+\epsilon_1(X_1^2+\cdots+X_n^2)\leq\delta_1^2\}
         \supseteq
\{x:F_L(X)^2+\epsilon_2
(X_1^2+\cdots+X_n^2)
\leq\delta_2^2\}
\supseteq\cdots$;
\item $\underset{i\in\mathbb{N}}{\bigcap}
\{x:F_L(X)^2+\epsilon_i
(X_1^2+\cdots+X_n^2)
\leq\delta_i^2\}
=A\cap L\cap B_{\leq r}(0)$;
\item Each $\delta_i^2$ is a regular value for $F_L(X)^2+\epsilon_i
(X_1^2+\cdots+X_n^2)$.
     \end{itemize}
     By Theorem \ref{bddcsmconn}, there is $N\in\mathbb{N}$ such that
     for all $i\in\mathbb{N}$,
     \[
    \sum {rank}\,\, H^i(\{x:F_L(X)^2+\epsilon_i
(X_1^2+\cdots+X_n^2)\leq\delta_i^2\})\leq  N.     \]
Hence by continuity of  C$\check{\text{e}}$ch cohomology (\cite[Chapter~X, Theorem~3.1]{eilenberg2015foundations} or \cite[6.18~Continuity]{dold2012lectures}),
\[\sum 
{rank}\,\,H^i(A\cap L\cap B_{\leq r}(0))\leq N.
\]
 \end{proof}
 \begin{theorem}\label{bddgamma}
     Given a quantifier-free  set $A\subseteq\mathbb{R}^n$ definable in $\mathbb{R}_{an,\exp,\varphi}$,
     if the regularity assumption in Theorem \ref{rtb} holds, then
    $\gamma(A)<\infty$.
 \end{theorem}
 \begin{proof}
Let $A\subseteq\mathbb{R}^n$ be a quantifier-free definable set.
Let $L$ be an affine subspace in $\mathbb{R}^n$.
By Theorem \ref{bddcomp} and Fact \ref{dual},
there is $N\in\mathbb{N}$ such that  for all $r\in\mathbb{N}$, 
\[
{rank} \,\, H_0(A\cap L\cap B_{\leq r}(0))\leq N,
\]
Hence, by Fact \ref{drlim}, we conclude that  $rank\,\,H_0(A\cap L)\leq N$.
 \end{proof}
 \subsubsection{Main Theorem}
\indent

 Finally, we conclude that the regularity assumption in Theorem \ref{rtb} implies  the o-minimality of $\mathbb{R}_{an,\exp,\varphi}$.
 \begin{theorem}\label{mainthm}
   Suppose that we have the regularity assumption in Theorem \ref{rtb}.

   Then
$\mathbb{R}_{an,\exp,\varphi}$ is o-minimal.
 \end{theorem}
 \begin{proof}
     Immediate from Theorem \ref{bddgamma} and Fact \ref{omintest}.
 \end{proof}
 \begin{remark}
     To prove o-minimality of  $\mathbb{R}_{an,\exp,\varphi}$, we only need a uniform bound for the \# of non-singular zeroes for systems of the form in Definition \ref{qfaff}.
     But once we have o-minimality, we have a uniform bound for the \# of non-singular zeroes for any form of definable systems.
This  gives a motivation for studying o-minimality.
 \end{remark}
    \newpage
\ignore{
\section{Questions}
\begin{itemize}
    \item 
More than just finding a trans-exponential o-minimal structure:
What about adding a smooth solution to some Abel equation/dynamic system?
Does it change o-minimality?
\item 
Conversely, how do dynamic systems behave in an o-minimal structure?
\item Suppose we add a smooth function $f:\mathbb{R}^2\rightarrow\mathbb{R}$ (What about $\mathbb{R}^n$ in general?)
where there is $N\in\mathbb{N}$ such that 
for all $t\in \mathbb{R}$,
$f(t,-)$ has $\leq N$ many change of monotonicity,
then adding would not change o-minimality?
Also need to consider the rate of change relative to polynomials?
\item Develop a theory similar to that in \cite{khovanskiui1991fewnomials},
but for solutions to Abel equations instead of for Pfaffian functions?
\item Characterize o-minimality using monodromy?
\item Does the structure $(\mathbb{R},<,+,\cdot,\varphi)$ have quantifier-elimination?
If not, what about model-completeness?
\item Effecitivity of the theory?
\item Relate this problem to the Lindemann-Weierstrass type theorems.
\item It seems o-minimality says much more than algebraic independence of definable functions.
\item Fibration, trivialization in o-minimal geometry.
\end{itemize}
\newpage
\section{Appendix: Understanding the Algebraic Topological Tools}\label{appendix}}

\bibliographystyle{alpha}
\bibliography{references}
\end{document}